\documentclass[11pt,reqno]{amsart}
\usepackage{a4wide}
\usepackage{hyperref,amsmath,amsfonts,mathscinet,amssymb,amsthm,natbib}
\usepackage{accents}
\usepackage{paralist}
\usepackage{constants}
\usepackage{xcolor}
\usepackage{tikz-cd}
\usepackage{todonotes}
\usepackage[mathscr]{euscript}
\usepackage{enumitem}

\usepackage{mathtools}

\theoremstyle{plain}
\newtheorem{thm}{Theorem}[section]
\newtheorem{lemma}[thm]{Lemma}
\newtheorem{cor}[thm]{Corollary}
\newtheorem{prop}[thm]{Proposition}

\theoremstyle{definition}
\newtheorem{rem}[thm]{Remark}
\newtheorem{cnv}[thm]{Convention}
\newtheorem{example}[thm]{Example}
\newtheorem{defn}[thm]{Definition}

\expandafter\let\expandafter\oldproof\csname\string\proof\endcsname
\let\oldendproof\endproof
\renewenvironment{proof}[1][\proofname]{%
  \oldproof[\bf #1]%
}{\oldendproof}

\newcommand{\abs}[1]{\left|{#1}\right|}
\newcommand{\norm}[1]{\left\lVert{#1}\right\rVert}

\newcommand{\R}{\mathbb{R}}

\newcommand{\M}{\mathcal{M}}
\newcommand{\K}{\mathcal{K}}

\DeclareMathOperator*{\esssup}{ess\,sup}
\DeclareMathOperator*{\essinf}{ess\,inf}
\DeclareMathOperator*{\loc}{loc}
\DeclareMathOperator*{\sgn}{sgn}

\begin{document}

\title{On optimal endpoints for integral kernel operators}

\author{Ivan Kotal\'ik}

\email[I.~Kotal\'ik]{ivan.kotalik690@student.cuni.cz}
\urladdr{0009-0009-0128-2111}

\address{
 Department of Mathematical Analysis,
	Faculty of Mathematics and Physics,
	Charles University,
	Sokolovsk\'a~83,
	186~75 Praha~8,
	Czech Republic}


\begin{abstract}
We study the behavior of integral kernel operators acting on functions of a single real variable. 
In particular, we attempt to find possible candidates for their optimal endpoint spaces in the class of Banach function spaces endowed with a rearrangement-invariant quasinorm. 
To achieve this, we characterize the existence of a Calder\'on-type estimate for these operators. We also build a theory of 
optimal endpoint spaces for the Calderón-type operator given by this pointwise estimate. We use these results to propose optimal endpoint spaces for some integral kernel operators.  
Furthermore, we show that these endpoints are, in fact, optimal under certain conditions.
\end{abstract}

\date{\today}

\maketitle

\bibliographystyle{abbrv}

\section{Introduction}

Integral operators play an important role in mathematical analysis and its applications. A pivotal example of such an operator is the well-known 
Laplace transform, defined for any appropriate $f$ on $(0,\infty)$ by
\begin{equation*}
    \mathcal{L} f(t) := \int_0^\infty f(s) e^{-st} \, ds.
\end{equation*}
Optimal endpoints for $\mathcal{L}$ and an upper estimate for the nonincreasing rearrangement of $\mathcal{L}$, namely
\begin{equation*}
    (\mathcal{L}f)^*(t)\leq \int_0^\frac{1}{t} f^*, \ t\in (0,\infty),
\end{equation*}
for every $f\in L^1 + L^\infty$, have already been 
found in~\cite{buriankova2017optimal}. Another important example of such an operator is the fractional maximal operator $M_\gamma$, where $\gamma\in [0,n)$,
defined for $f\in L^1_{\loc}(\R^n)$ by
\begin{equation*}
    (M_\gamma f)(x) := \sup_{Q\ni x} \abs{Q}^{\gamma /n-1} \int_Q \abs{f}, \ x\in \R^n,
\end{equation*}
where the supremum is extended over all cubes $Q\subset \R^n$ with sides parallel to the coordinate axes, and $\abs{Q}$ denotes 
the $n$-dimensional Lebesgue measure. An upper estimate for the nonincreasing rearrangement of $M_\gamma$, namely 
\begin{equation*}
    (M_\gamma f)^*(t)\leq C \sup_{t<\tau<\infty} \tau^{\gamma /n-1} \int_0^\tau f^*, \ t\in (0,\infty),
\end{equation*}
for every $f\in L^1_{\loc}(\R^n)$ and a constant $C \in (0,\infty)$ depending only on $\gamma$ and $n$, has been shown in~\cite{MR1758860}. This type of estimate is called a Calder\'on-type estimate. More precisely, let $T$ be an operator acting on $\M(R,\mu)$ for some general measure space and $S$ be an operator acting on $\M((0,\infty),\lambda)$. Let there exist a $C\in (0,\infty)$ such that it holds, for all appropriate $f$,
\begin{equation*}
    (Tf)^*(t)\leq C S (f^*)(t), \ t\in(0,\infty).
\end{equation*}
This inequality is called a Calder\'on-type estimate for $T$ and $S$ is called a Calder\'on-type operator. 
Further examples of these estimates could be found for the Hardy-Littlewood maximal operator and the Hilbert transform in~\cite[Chapter 3, Theorem 3.8]{bs88} and \cite[Chapter 3, Theorem~4.8]{bs88}, respectively. \par
 Regarding the usefulness of Calder\'on-type estimates, consider the Riesz potential, defined as 
\begin{equation*}
    I_\alpha f(x) = \int_{\R^n} \frac{f(y)}{\abs{x-y}^{n-\alpha}} \, dy, \ \alpha\in (0,n) , \ x\in \R^n,
\end{equation*}
for any appropriate $f$, and consider the Poisson equation 
\begin{equation*}
    -\Delta u = g,
\end{equation*}
for a given function $g$. It is classical that, for a sufficiently smooth $g$, a solution to this equation can be written as $u=C I_2 g$ for $n\geq 3$, 
where $C\in (0,\infty)$ depends only on $n$ (see~\cite[Chapter 2, Theorem 1]{MR1625845}). We may use the Calder\'on-type estimate for $I$, which is easily derived from~\cite[Theorem 1.7]{MR146673}, 
to transfer regularity from $g$ to $u$. \par
The main focus of this paper is to greatly extend and generalize the results from~\cite{buriankova2017optimal}. We are interested in studying integral kernel operators of the form
\begin{equation*}
    Kf(t):=\int_{0}^{\infty} f(s) k(s,t) \, ds, \ t\in (0,\infty),
\end{equation*}
where $k$ is a nonnegative function that is nonincreasing in the second variable. In Section~\ref{prelims}, we establish notation and equip the reader with the necessary knowledge to continue reading. In Section~\ref{calderon}, we 
present a Calderón-type estimate for these integral operators, and 
we also characterize all the kernels $k$ for which this estimate exists. In Section~\ref{variable_integral}, this motivates us to study the behavior of integral operators of the form 
\begin{equation*}
    T_bf(t):=\int_0^{b(t)} f^*,
\end{equation*}
where $b$ is a nonnegative function, on (quasinormed) rearrangement-invariant spaces. Lastly, in Section~\ref{examples}, 
we attempt to apply the established theory to some specific
integral kernel operators. In particular, we attempt to find optimal endpoints for these operators. Note that the question of optimal domains for integral kernel operators has already been studied using different methods, for example, in~\cite{Curbera} and~\cite{Ricker}.

\section{Preliminaries}\label{prelims}
We shall recall some basic properties of function spaces endowed with rearrangement-invariant quasinorms and notions we need from interpolation 
theory. More in-depth resources on these topics include \cite{bs88}, \cite{MR482275}, \cite{pick2012function}, \cite{Nekvinda_2024}, and \cite{propertiesrearrangement}. \par
We define
\begin{equation*}
    \mathcal{M}(0,\infty):=\{f\colon (0,\infty)\rightarrow [-\infty, \infty] : f \textrm{ is Lebesgue measurable on } (0,\infty)\},
\end{equation*}
\begin{equation*}
    \mathcal{M}_{+}(0,\infty):=\{f\in \mathcal{M}(0,\infty) : f \in [0,\infty]\},
\end{equation*}
and
\begin{equation*}
    \mathcal{M}_0(0,\infty):=\{f\in \mathcal{M}(0,\infty) : f \textrm{ is finite $\lambda$-a.e.}\}.
\end{equation*}
We denote the Lebesgue measure of any Lebesgue measurable set $M\subset \R$ as $\abs{M}$ or $\lambda(M)$. \par
The \textsl{distribution function $\mu_f$} of 
$f\in\mathcal{M}(0,\infty)$ is defined as
\begin{equation*}
    \mu_f(t):=\lambda(\{x\in (0,\infty): \abs{f(x)}>t\}), \ t\in [0,\infty),
\end{equation*}
and the \textsl{nonincreasing rearrangement} of $f\in\mathcal{M}(0,\infty)$ is the function $f^*$ defined as 
\begin{equation*}
    f^*(x):=\inf\{t \in [0,\infty): \mu_f(t) \leq x\}, \ x\in [0,\infty).
\end{equation*}
The operation $f\mapsto f^*$ is monotone, that is, $\abs{f}\leq\abs{g}$ a.e.\ implies $f^*\leq g^*$, for each $f$, $g\in\mathcal{M}(0,\infty)$. 
Unfortunately, it is not subadditive. However, the following two inequalities hold
\begin{equation}\label{fake_subaditivity}
    \int_0^t (f+g)^*(s)\, ds \leq \int_0^t f^*(s) \, ds + \int_0^t g^*(s)\, ds,
\end{equation}
\begin{equation}\label{rear_ineq}
    (f+g)^*(t_1 + t_2) \leq f^*(t_1) + g^*(t_2),
\end{equation}
for every $t$, $t_1$, $t_2\in [0,\infty)$ and $f$, $g\in \mathcal{M}_0(0,\infty)$. Another essential property of this operation is the \textsl{Hardy-Littlewood inequality}. 
It states that if $f$, $g\in\mathcal{M}(0,\infty)$, then
\begin{equation*}
    \int_0^\infty \abs{f(t)g(t)} \, dt \leq \int_0^\infty f^*(t)g^*(t) \, dt, \ t\in (0,\infty).
\end{equation*} 
Here and in the rest of this paper, we adopt the convention that $0\cdot\infty = 0$.
We also note that 
\begin{equation}\label{rear_l_infty}
    f^*(0)=\esssup_{t\in (0,\infty)} \abs{f(t)}, \ f\in \M(0,\infty).
\end{equation} 
For $f\in \M(0,\infty)$, we also define the \textsl{maximal function} of $f^*$ as 
\begin{equation*}
    f^{**}(t):=\frac{1}{t}\int_0^t f^*, \ t\in (0,\infty).
\end{equation*}
An intrinsic property of the maximal function is that, for any $f\in \M(0,\infty)$, we have $f^*\leq f^{**}$.\par
A functional $\rho\colon \mathcal{M}_+(0,\infty)\rightarrow [0,\infty]$ satisfying:
\begin{enumerate}[label=(P\arabic*), align=left]
    \item $\rho(f)=0$ if and only if $f\equiv0$; $\rho(\lambda f)=\lambda \rho(f)$; $\rho(f+g)\leq \rho(f)+\rho(g)$;
    \item $f\leq g$ a.e.\ implies $\rho(f)\leq \rho(g)$;
    \item $f_j \nearrow f$ a.e.\ implies $\rho(f_j)\nearrow \rho(f)$;
    \item $\rho(\chi_G)<\infty$ for every $G\subset (0,\infty)$ of finite measure;
    \item $\int_G f(t) \, dt \leq C_G\rho(f)$ for some $C_G$ and every $G\subset(0,\infty)$ of finite measure;
    \item $\rho(f)=\rho(g)$ whenever $f^*=g^*$,
\end{enumerate}
for all $f$, $g$, and $(f_j)_{j=1}^\infty$ in $\mathcal{M}_+(0,\infty)$, and every $\lambda\in [0,\infty)$, is called a \textsl{rearrangement-invariant norm}. 
Let $\rho$ be a rearrangement-invariant norm. The collection $X:=\{f\in\mathcal{M}(0,\infty): \rho(\abs{f})<\infty\}$ is called a \textsl{rearrangement-invariant space}. 
We also define a more general \textsl{rearrangement-invariant quasinorm} as a functional $\rho\colon \mathcal{M}_+(0,\infty)\rightarrow [0,\infty]$ satisfying (P2), (P3), (P4), (P6), and 
\begin{enumerate}[label=(Q\arabic*), align=left]
    \item $\rho(f)=0 \textrm{ if and only if } f\equiv0\textrm{; } \rho(\lambda f)=\lambda \rho(f)\textrm{; } \rho(f+g)\leq C(\rho(f)+\rho(g))$,
\end{enumerate}
for every $f\in \mathcal{M}_+(0,\infty)$, $\lambda\in [0,\infty)$, and a constant $C\in[1,\infty)$. The smallest such $C$ is called the \textsl{modulus of concavity}.
Analogously, we define a \textsl{quasinormed rearrangement-invariant space}. We note that all rearrangement-invariant norms are also 
rearrangement-invariant quasinorms and all rearrangement-invariant spaces are also quasinormed rearrangement-invariant spaces. From (Q1) and (P2)-(P4), it can be easily observed that every $f\in\M_+(0,\infty)$ 
with $\rho(f)<\infty$, where $\rho$ is a rearrangement-invariant quasinorm, is finite a.e.
For each $f\in \mathcal{M}(0,\infty)$, we define 
\begin{equation*}
    \norm{f}_X:=\rho(\abs{f}),
\end{equation*}
where $\rho$ is the corresponding rearrangement-invariant quasinorm of $X$. We say that two rearrangement-invariant quasinorms $\rho_1$, $\rho_2$ are \textsl{equivalent}, denoted by 
$\rho_1\approx\rho_2$, if there exist constants $c_1$, $c_2\in (0,\infty)$ such that
\begin{equation*}
    c_1\rho_1(f)\leq\rho_2(f)\leq c_2\rho_1(f),
\end{equation*}
for every $f\in \mathcal{M}_+(0,\infty)$. We say that two quasinormed rearrangement-invariant spaces $X$, $Y$ are \textsl{equivalent}, denoted by $X\approx Y$, if their quasinorms are equivalent.
For a bounded linear operator $T$ between quasinormed rearrangement-invariant spaces X and Y, denoted by $T\colon X\to Y$, we define the \textsl{operator norm} of $T$ as
\begin{equation*}
    \norm{T}_{X\rightarrow Y}:=\sup\left\{\norm{Tf}_Y : f \in X \textrm{, } \norm{f}_X\leq 1\right\}.
\end{equation*} 
For better readability, we will use this notation for suprema
\begin{equation*}
    \sup\left\{\norm{Tf}_Y : f \in X \textrm{, } \norm{f}_X\leq 1\right\} = \sup_{\norm{f}_X\leq 1} \norm{Tf}_Y.
\end{equation*} \par
For instance, the well-known $L^p$ spaces with $p\in (0,\infty]$ are quasinormed rearrangement-invariant spaces. Another example of quasinormed rearrangement-invariant spaces are the \textsl{classical Lorentz spaces}. 
Let $p\in (0,\infty)$, and let $w\in\mathcal{M}_+(0,\infty)$. The space $\Lambda^p(w)$ is a space defined by the functional 
\begin{equation*}
    \norm{f}_{\Lambda^p(w)}:=\left(\int_0^\infty (f^*(t))^p w(t) \, dt\right)^\frac{1}{p}, \ f\in \M(0,\infty),
\end{equation*}
and the space $\Gamma^p(w)$ is a space defined by the functional 
\begin{equation*}
    \norm{f}_{\Gamma^p(w)}:=\left(\int_0^\infty (f^{**}(t))^p w(t) \, dt\right)^\frac{1}{p}, \ f\in \M(0,\infty).
\end{equation*} 
Note that not every combination of $p$ and $w$ yields a rearrangement-invariant quasinorm. 
An important special case of the $\Lambda$ spaces are the \textsl{Lorentz spaces}. Let $w(t):=t^{q/p-1}$, $t\in (0,\infty)$, where $p$, $q\in (0,\infty)$. We then define the Lorentz space 
$L^{p,q}$ with the following quasinorm
\begin{equation*}
    \norm{f}_{L^{p,q}}:=\norm{f}_{\Lambda^q{(w)}}, \ f\in \M(0,\infty).
\end{equation*}
Furthermore, we consider the \textsl{intersection} and the \textsl{sum} of quasinormed rearrangement-invariant spaces $X$, $Y$ equipped with the quasinorms 
\begin{equation*}
    \norm{f}_{X \cap Y}:=\max\{\norm{f}_X,\norm{f}_Y\}, \ f\in \M(0,\infty),
\end{equation*}
and 
\begin{equation*}
    \norm{f}_{X+Y}:=\inf_{f=g+h}(\norm{g}_X + \norm{h}_Y), \ f\in \M(0,\infty),
\end{equation*}
where the infimum is taken over all decompositions $f=g+h$ such that $g\in X$ and $h\in Y$. \par
A quasinormed rearrangement-invariant space $Y$ is a \textsl{range partner} for a given quasinormed rearrangement-invariant space $X$ with respect to a linear operator $T$ 
if $T\colon X\rightarrow Y$. We say that $Y$ is the \textsl{optimal range partner} for $X$ with respect to $T$ if one has $Y\hookrightarrow Z$, for every range partner $Z$ for $X$ with 
respect to $T$. Similarly, we define a \textsl{domain partner} with respect to $T$ and the \textsl{optimal domain partner} with respect to $T$, i.e.,\ the largest possible domain space. 
We may require these quasinormed rearrangement-invariant spaces to have additional properties. In that case, optimality is restricted to spaces with these properties. \par
For each $a\in (0,\infty)$, let $D_a$ denote the \textsl{dilation operator} defined by
\begin{equation*}
    (D_af)(t):=f(at),
\end{equation*}
for every $f\in\mathcal{M}_+(0,\infty)$ and $t\in (0,\infty)$. This operator is bounded on every quasinormed rearrangement-invariant space (see \cite[Theorem 3.23]{Nekvinda_2024}). 
Every quasinormed rearrangement-invariant space $X$ has its so-called \textsl{associate space} denoted as $X'$, equipped with the norm
\begin{equation*}
    \norm{f}_{X'}:=\sup\left\{\int_0^\infty \abs{fg} : g\in X \textrm{, }\norm{g}_X\leq1\right\}.
\end{equation*}
If $X$ has the property (P5), the space $X'$ is a nontrivial rearrangement-invariant space. This has been observed earlier in \cite[Remark 2.3]{Optimal_Sobolev}.
The norm $\norm{\cdot}_{X'}$ can also be expressed as
\begin{equation}\label{abs_out}
    \norm{f}_{X'}=\sup\left\{\int_0^\infty f^*g^* : g\in X \textrm{, }\norm{g}_X\leq1\right\}
\end{equation}
(proof is similar to that of \cite[Chapter 2, Proposition 2.4]{bs88}).
If a quasinormed rearrangement-invariant space $X$ has the property (P5), then
\begin{equation}\label{associate_ineq}
    \norm{f}_{X''}\leq \norm{f}_X,
\end{equation}
for any $f\in \M (0,\infty)$ (the proof is analogous to that of~\cite[Chapter 1, Theorem 2.7]{bs88}). 
Furthermore, if $X$ is a rearrangement-invariant space,~\eqref{associate_ineq} turns into an equality (see~\cite[Chapter~1, Theorem 2.7]{bs88}).\par
Let $X$, $Y$ be quasinormed rearrangement-invariant spaces and $f\in X+Y$. We define the \textsl{Peetre $K$-functional}, often abbreviated as $K$-functional, by
\begin{equation*}
    \K(f,t;X,Y)=\inf_{f=g+h}(\norm{g}_X + t \norm{h}_Y), \ t\in (0,\infty),
\end{equation*}
where the infimum is taken over all decompositions $f=g+h$ such that $g\in X$ and $h\in Y$. A trivial property of the $K$-functional is that
\begin{equation}\label{K_swap}
    \K(f,t;X,Y)=t\K\left(f,t^{-1};Y,X\right), \ t\in (0,\infty).
\end{equation}
The $K$-functional behaves well if we study it as a function of a real variable. Indeed, the function $\K(f,\cdot;X,Y)$ defined on $(0,\infty)$ 
is positive, increasing, and concave. 
It is known (see \cite[Chapter 5, Theorem 1.6]{bs88}) that
\begin{equation}\label{l1_linfty_K}
    \K(f,t;L^1,L^\infty)=\int_0^t f^*, \ t\in (0,\infty).
\end{equation}
Let $X$, $\tilde{X}$, $Y$, $\tilde{Y}$ be quasinormed rearrangement-invariant spaces and $T$ a linear operator such that 
$T\colon X\rightarrow \tilde{X}$ and $T\colon Y \rightarrow \tilde{Y}$. Define $C:=\max\{\norm{T}_{X\rightarrow \tilde{X}},\norm{T}_{Y\rightarrow \tilde{Y}}\}$. 
Then we have
\begin{equation}\label{K_property}
    \K(Tf,t;\tilde{X},\tilde{Y}) \leq C \K(f,t;X,Y), \ t\in (0,\infty).
\end{equation}
This elementary result, which can be easily proven from the definition, 
is the cornerstone of the interpolation properties of the $K$-functional. For example, given any $q\in (1,\infty)$ and $\theta \in (0,1)$,  
we have
\begin{equation}\label{real_interpolation}
    \norm{\K(Tf,t;\tilde{X},\tilde{Y}) t^{-\theta-\frac{1}{q}}}_{L^q} \leq C \norm{\K(f,t;X,Y) t^{-\theta-\frac{1}{q}}}_{L^q},
\end{equation}
where the $L^q$ norms are measured in the variable $t$. Notice that the spaces $X$, $Y$, $\tilde{X}$, $\tilde{Y}$ are not strictly required to be quasinormed rearrangement-invariant spaces. In this paper, we leverage this fact. \par
Let $b\colon (0,\infty)\rightarrow[0,\infty)$ be a nonincreasing function that is not identically zero. We define the \textsl{generalized inverse of $b$} as
    \begin{equation*}
        b^{-1}(t):=\begin{cases}    \sup\{s\in (0,\infty): b(s)\geq t\} & \textrm{if } t\in (0,b(0)),\\
                                    0 & \textrm{if } t\in [b(0),\infty).
                    \end{cases} 
    \end{equation*}
Throughout the paper, the symbol $b^{-1}$ will always denote the generalized inverse of $b$. \par

\section{Calderón-type estimate for integral kernel operators}\label{calderon}
In this section, we analyze a Calder\'on-type estimate for integral kernel operators and provide sufficient and necessary conditions required on the kernel of the operator
for this estimate to exist. We begin with a couple of key definitions, which we use for the rest of this section.

\begin{defn}\label{m_phi}
Let $\varphi\colon (0,\infty)\rightarrow (0,\infty)$ be a nondecreasing function. We define the space $m_\varphi$ as
    \begin{equation*}
        m_\varphi := \left\{f\in \M(0,\infty): \esssup_{t\in (0,\infty)} \varphi(t) f^*(t)<\infty\right\},
    \end{equation*}
    and
    \begin{equation*}
        \norm{f}_{m_\varphi}:=\esssup_{t\in (0,\infty)} \varphi(t) f^*(t), \ f\in \M(0,\infty).
    \end{equation*}
\end{defn}

    The space $m_\varphi$ is analogous in definition to the space $\Lambda^\infty(w)$ discussed in~\cite[Chapter 10]{pick2012function}. 
    This leads us to conclude that it generally may not be a quasinormed rearrangement-invariant space. 
    In fact, it may not even be a linear space (see~\cite[Theorem 1.6]{classical_lorentz}). 
    It is easy to see that the space at least satisfies (P4). Thus, it is nontrivial (not equal to $\{0\}$).
    We will still treat it as a quasinormed rearrangement-invariant space regardless, as it 
    simplifies the notation, and all the properties of quasinormed rearrangement-invariant spaces, namely, the triangle inequality, are not needed. \par

\begin{defn}
    We say that a measurable function $k\colon (0,\infty) \times (0,\infty) \rightarrow [0,\infty]$ is an admissible kernel if $k(\cdot,t)$ is integrable for a.e.\ $t\in (0,\infty)$. If $k$ is decreasing (nonincreasing) in the second variable, we call it a decreasing (nonincreasing) admissible kernel. 
    Let $k$ be an admissible kernel. We define the integral kernel operator $K_k$ as 
    \begin{equation*}
        K_k f(t):=\int_0^\infty f(s) k(s,t) \, ds, \ t\in (0,\infty),
    \end{equation*}
    for every $f\in \M(0,\infty)$ for which the integral makes sense.
\end{defn}

    To obtain our desired estimate, we will employ the $K$-functional,
    whose application will require that the operator in question acts
    boundedly between certain two pairs of `endpoint' spaces. In the following proposition, we show sufficient conditions on the kernel $k$ that give us this boundedness property.

\begin{prop}\label{kernel_conditions}
    Let $k$ be an admissible kernel, and define $w(t):= \norm{k(\cdot,t)}_{L^1}$, $t\in (0,\infty)$.
    Let there exist $C_1$, $C_2\in (0,\infty)$ and $\varphi\colon (0,\infty)\rightarrow (0,\infty)$ nondecreasing 
    such that
    \begin{equation}\label{L1_est}
        w^*(t)\leq \frac{C_1}{\varphi(t)}, \ \text{a.e.}\ t\in (0,\infty),
    \end{equation}
    \begin{equation}\label{Linf_est}
        \norm{k(\cdot,t)}_{L^\infty}\leq C_2, \ \text{a.e.} \ t\in (0,\infty).
    \end{equation} 
    Then $K_k\colon L^\infty \rightarrow m_\varphi$, $K_k\colon L^1\rightarrow~L^\infty$, $\norm{K_k}_{L^\infty \rightarrow m_\varphi}\leq C_1$, and 
    $\norm{K_k}_{L^1\rightarrow L^\infty}\leq C_2$.
\end{prop}
\begin{proof}
    Choose $g\in L^\infty$. 
    Using~\eqref{L1_est}, we get
    \begin{equation*}
        (K_kg)^*(t)=\left(\int_{0}^{\infty} g(s)k(s,\cdot)\,ds\right)^*(t)\leq \norm{g}_{L^\infty} \frac{C_1}{\varphi(t)}, \ \text{a.e.}\ t\in (0,\infty).
    \end{equation*}
    Multiplying by $\varphi(t)$ and taking the essential supremum over $t\in (0,\infty)$ on the left-hand side gives us $K_k\colon L^\infty \rightarrow m_\varphi$.
    Next, choose $h\in L^1$. Using~\eqref{Linf_est} and the H\"older inequality, we get
    \begin{equation*}
        \abs{K_k h(t)}\leq\int_{0}^{\infty} \abs{h(s)}k(s,t)\,ds \leq \norm{k(\cdot,t)}_{L^\infty} \norm{h}_{L^1} \leq C_2 \norm{h}_{L^1}, \ \text{a.e.}\ t\in (0,\infty).
    \end{equation*}
    This shows that $K_k\colon L^1\rightarrow L^\infty$. The estimates for the norms of $K_k$ are an immediate consequence 
    of the proven inequalities.
\end{proof}

\begin{rem}
    For the Laplace transform defined by $k(s,t):=e^{-st}$, we obtain $w(t)=\frac{1}{t}$ and $C_1=C_2=1$. This is 
    consistent with~\cite{buriankova2017optimal}.
\end{rem}

The next step in obtaining the estimate is to calculate the involved $K$-functional. 
The following proposition gives a characterization of this $K$-functional. While we only require the lower bound, we provide 
a full characterization for completeness. This result was already obtained for more general spaces in~\cite{Maligranda} and~\cite{Torchinsky}. However, these publications do not contain a calculation of the optimal constants, so we provide it here along with a full proof of the estimates for the reader's convenience.
\begin{prop}\label{inf_inequality}
    Let $\varphi\colon (0,\infty)\rightarrow (0,\infty)$ be a nondecreasing function and $f\in m_\varphi+L^\infty$. Then
    \begin{equation}\label{lower_k}
        \norm{\chi_{(0,t)}f^*}_{m_\varphi}\leq\K(f,\varphi(t);m_\varphi,L^\infty), \ t\in (0,\infty),
    \end{equation}
    and
    \begin{equation}\label{upper_k}
        \K(f,\varphi(t);m_\varphi,L^\infty)\leq 2\norm{\chi_{(0,t)}f^*}_{m_\varphi}, \ \text{a.e.}\ t\in (0,\infty),
    \end{equation}
    where~\eqref{upper_k} holds for every $t\in (0,\infty)$ if $\varphi$ is left-continuous.
    Moreover, the constant 1 in the inequality~\eqref{lower_k} is optimal. If $\norm{\varphi}_{L^\infty}=\infty$ and $\lim_{s\rightarrow0+}\varphi(s)=0$, 
    the constant 2 in~\eqref{upper_k} is also optimal.
\end{prop}

\begin{proof}
    We start with the inequality~\eqref{lower_k}. Choose $t\in (0,\infty)$. Using~\eqref{rear_ineq},~\eqref{rear_l_infty}, and the fact that $\varphi$ 
    is nondecreasing, we get, for a.e.\ $s\in (0,t)$,
    \begin{equation*}
        \begin{aligned}
            \K(f,\varphi(t);m_\varphi,L^\infty)&=\inf_{f=g+h} \left(\esssup_{u\in (0,\infty)} \varphi(u) g^*(u) + \varphi(t) \norm{h}_{L^\infty}\right) \\
            &\geq \inf_{f=g+h} \left(\varphi(s) g^*(s) + \varphi(s) \norm{h}_{L^\infty}\right)\\
            &=\inf_{f=g+h} \left(\varphi(s) g^*(s) + \varphi(s) h^*(0)\right) \\
            &=\varphi(s)\inf_{f=g+h} \left( g^*(s) + h^*(0)\right) \\
            &\geq \varphi(s)\inf_{f=g+h} \left( (g+h)^*(s)\right) = \varphi(s) f^*(s).
        \end{aligned}
    \end{equation*}
    Taking the essential supremum over $(0,t)$ yields $\norm{\chi_{(0,t)}f^*}_{m_\varphi}\leq \K(f,\varphi(t);m_\varphi,L^\infty)$. 
    This inequality also shows that $\norm{\chi_{(0,t)}f^*}_{m_\varphi}$ is finite for every $t\in (0,\infty)$ because $f\in m_\varphi+L^\infty$, 
    which makes the $K$-functional finite. To prove~\eqref{upper_k}, we note that for the set $X:=\{t\in (0,\infty) : \varphi f^* \text{ is continuous at }t\}$ we 
    have $\abs{(0,\infty)\setminus X}=0$.
    Choose $t\in X$ and define 
    \begin{equation*}
        h(s):=\min\{\abs{f(s)},f^*(t)\} \sgn{(f(s))}, \ s\in (0,\infty),
    \end{equation*}
    \begin{equation*}
        g(s):=(\abs{f(s)}-h(s))\sgn{(f(s))}=\max\{\abs{f(s)}-f^*(t),0\} \sgn{(f(s))}, \ s\in (0,\infty).
    \end{equation*}
    Then $f=g+h$. We observe that $\mu_{\min\{\abs{f},f^*(t)\}}=\mu_{\min\{f^*,f^*(t)\}}$ and $\mu_{\max\{\abs{f}-f^*(t),0\}}=\mu_{\max\{f^*-f^*(t),0\}}$. Thus, 
    \begin{equation*}
        h^*(s):=\min\{f^*(s),f^*(t)\}, \ s\in (0,\infty),
    \end{equation*}
    \begin{equation*}
        g^*(s):=\max\{f^*(s)-f^*(t),0\}=(f^*(s)-f^*(t))\chi_{(0,t)}(s), \ s\in (0,\infty).
    \end{equation*}
    Because $t\in X$, we have $\varphi(t)f^*(t)\leq \esssup_{s\in (0,t)} \varphi(s)f^*(s)$. We now proceed with the following calculation 
    \begin{equation*}
        \begin{aligned}
            \K(f,\varphi(t);m_\varphi,L^\infty)&\leq \norm{g}_{m_\varphi} + \varphi(t) \norm{h}_{L^\infty} \\
            &= \esssup_{s\in (0,t)} \varphi(s)(f^*(s)-f^*(t)) + \varphi(t) f^*(t) \\
            &\leq \esssup_{s\in (0,t)} \varphi(s)f^*(s) + \varphi(t)f^*(t) \\
            &\leq 2\esssup_{s\in (0,t)} \varphi(s)f^*(s) = 2 \norm{\chi_{(0,t)}f^*}_{m_\varphi}.
        \end{aligned}
    \end{equation*}
    If $\varphi$ is left-continuous, the function $\K(f, \varphi(\cdot);m_\varphi, L^\infty)$ is left-continuous because the $K$-functional 
    is continuous as a function of a real variable. Choose $t\in (0,\infty)$ and $(t_n)_{n=1}^\infty$ in $X$ such that $t_n\nearrow t$. Then $\K(f, \varphi(t_n);m_\varphi, L^\infty)\nearrow 
    \K(f, \varphi(t);m_\varphi, L^\infty)$ and $2\norm{\chi_{(0,t_n)}f^*}_{m_\varphi}=2\norm{\chi_{(0,t_n)}\varphi f^*}_{L^\infty}\nearrow 
    2\norm{\chi_{(0,t)}\varphi f^*}_{L^\infty}=2\norm{\chi_{(0,t)}f^*}_{m_\varphi}$, which extends the inequality to all $t\in (0,\infty)$. \par
    It remains to show the optimality of the constants 1 and 2. For~\eqref{lower_k}, consider the constant 
    function $f=\lim_{t\rightarrow 1_-}1/\varphi(t)$. Then $\norm{\chi_{(0,1)}f^*}_{m_\varphi}=1$. The trivial decomposition $g=0$ and 
    $h=f$ shows that $f$ saturates the inequality. For~\eqref{upper_k}, assume that $\norm{\varphi}_{L^\infty}=\infty$ 
    and $\lim_{s\rightarrow0+}\varphi(s)=0$. 
    We define 
    \begin{equation*}
        f(s):=\begin{cases} \frac{1}{\varphi(s)} & \text{if } s\in (0, 1], \\
                            \frac{1}{\varphi(1)} & \text{if } s\in (1,\infty).
            \end{cases}
    \end{equation*}
    Then $\norm{\chi_{(0,1)}f^*}_{m_\varphi}=1$. To prove $\K(f,\varphi(t);m_\varphi,L^\infty)\geq 2$, we will show that, for any decomposition 
    $f=g+h$ where $g\in m_\varphi$ and $h\in L^\infty$, it holds $\norm{g}_{m_\varphi}\geq 1$ 
    and $\norm{h}_{L^\infty}\geq 1/\varphi(1)$. Let $f=g+h$ be such a decomposition. Assume that there exists a $c \in (0,\infty)$ such 
    that $c=\norm{h}_{L^\infty}< 1/\varphi(1)$. Because $g=f-h$, 
    we have that $g\geq 1/\varphi(1) - c>0$. Then
    \begin{equation*}
        \norm{g}_{m_\varphi}= \esssup_{s\in (0,\infty)} \varphi(s) g^*(s)\geq \left(\frac{1}{\varphi(1)} - c\right) \norm{\varphi}_{L^\infty} = \infty.
    \end{equation*}
    This is a contradiction. Next, assume again that there exists a $c \in (0,\infty)$ such that $c=\norm{g}_{m_\varphi}< 1$. This implies that 
    $\varphi(s) g^*(s) \leq c$, for a.e.\ $s\in (0,\infty)$, which gives 
    \begin{equation*}
        g^*(s) \leq \frac{c}{\varphi(s)}, \ \text{a.e.}\ s\in (0,\infty).
    \end{equation*}
    Using~\eqref{rear_ineq}, we obtain that
    \begin{equation*}
        f^*(s)=(g+h)^*(s)\leq g^*(s) + \norm{h}_{L^\infty} \leq \frac{c}{\varphi(s)} + \norm{h}_{L^\infty},\ \text{a.e.}\ s\in (0,\infty).
    \end{equation*}
    However, this cannot be true as $f=f^*$ a.e.~and the function on the right-hand side is 
    smaller than $f$ near zero because $\lim_{s\rightarrow0+}\varphi(s)=0$. Consider the decomposition 
    $g=\chi_{(0,1]} f$ and $h=\chi_{(1,\infty)} f$. Then $f=g+h$, $\norm{g}_{m_\varphi}=1$, and $\norm{h}_{L^\infty}= 1/\varphi(1)$. 
    This finally shows that $\K(f,\varphi(1);m_\varphi,L^\infty)=2$, which proves the desired optimality.
\end{proof}

\begin{rem}
    The constant 2 in Proposition~\ref{inf_inequality} may not be optimal if the conditions imposed on $\varphi$ do not hold. Consider the constant function $\varphi=1$ 
    on $(0,\infty)$.
    Then $\norm{\cdot}_{m_\varphi}=\norm{\cdot}_{L^\infty}$, and it can be easily verified that
    \begin{equation*}
        \K(f,\varphi(t);m_\varphi,L^\infty)=\K(f,1;L^\infty,L^\infty)=\norm{f}_{L^\infty}=\norm{\chi_{(0,t)}f^*}_{m_\varphi},
    \end{equation*}
    for any $f\in L^\infty$, $t\in (0,\infty)$.
\end{rem}

Now, we use the previous result to obtain a Calder\'on-type estimate for our kernel operator.
\begin{thm}\label{kernel_estimate}
    Let $k$ be an admissible kernel, and let there exist a nondecreasing function $\varphi\colon (0,\infty)\rightarrow (0,\infty)$ 
    such that $K_k \colon L^\infty \rightarrow m_\varphi$ and $K_k \colon L^1\rightarrow L^\infty$. 
    Then for every $f\in L^1 + L^\infty$, we have
    \begin{equation}\label{estimate}
        (K_k f)^*(t)\leq C \int_0^{\frac{1}{\varphi(t)}} f^*, \ \text{a.e.}\ t\in (0,\infty),
    \end{equation}
    where $C=\max\{\norm{K_k}_{L^\infty \rightarrow m_\varphi}, \norm{K_k}_{L^1\rightarrow L^\infty}\}$. Furthermore, if $\varphi$ is right-continuous on $(0,\infty)$, the 
    inequality~\eqref{estimate} holds for every $t\in (0,\infty)$.
\end{thm}
\begin{proof}
    Fix $f\in L^1 + L^\infty$. By~\eqref{l1_linfty_K} and~\eqref{K_property}, we have
    \begin{equation*}
        \begin{aligned}
            \K(K_k f,t;L^\infty,m_\varphi) \leq C \K(f,t;L^1,L^\infty) = C \int_0^t f^*, \ t\in (0,\infty).
        \end{aligned}
    \end{equation*}
    We denote $X:={\{t\in (0,\infty) : \varphi \cdot (K_k f)^* \text{ is continuous at }t\}}$. 
    Then $\abs{(0,\infty)\setminus X}=0$, and $\varphi(t) (K_k f)^*(t)\leq \esssup_{0<s<t} \varphi(s) (K_k f)^*(s)$, $t\in X$.
    Combining these estimates with Proposition~\ref{inf_inequality} and~\eqref{K_swap}, we finally get
    \begin{equation*}
        \begin{aligned}
            (K_k f)^*(t)\leq \frac{1}{\varphi(t)} \K(K_k f,\varphi(t);m_\varphi,L^\infty) 
            =\K(K_k f,\frac{1}{\varphi(t)};L^\infty,m_\varphi) 
            \leq C \int_0^{\frac{1}{\varphi(t)}} f^*,
        \end{aligned}
    \end{equation*}
     for every $t\in X$. Now, when $\varphi$ is right-continuous on $(0,\infty)$, both sides of the inequality are right-continuous on $(0,\infty)$. 
     This means that we can extend the 
     inequality to all $t\in (0,\infty)$.
\end{proof}

We note that the requirements on the kernel $k$ in Proposition~\ref{kernel_conditions} are quite strong. To justify this, we will show that 
they are indeed necessary and cannot be avoided.
\begin{thm}\label{necessary_condtions}
    Let $k$ be an admissible kernel, and 
    let there exist a $C \in (0,\infty)$ and a nondecreasing function $\varphi\colon (0,\infty)\rightarrow (0,\infty)$ such that 
    \begin{equation*}
        (K_k f)^*(t)\leq C \int_0^{\frac{1}{\varphi(t)}} f^*, \ f\in L^1 + L^\infty, \ \text{a.e.}\ t\in (0,\infty).
    \end{equation*}
    Define $w(t):= \norm{k(\cdot,t)}_{L^1}$, $t\in (0,\infty)$. Then $k$ satisfies~\eqref{L1_est} and~\eqref{Linf_est}, where $C_1=C_2=C$. 
\end{thm}
\begin{proof}
    For~\eqref{L1_est}, put $f= 1$. Then $f\in L^1 + L^\infty$, and we get
    \begin{equation*}
        w^*(t)=(K_k f)^*(t)\leq \frac{C}{\varphi(t)}, \ \text{a.e.}\ t\in (0,\infty).
    \end{equation*}
    For~\eqref{Linf_est}, 
    let $E$ be a finite interval in $(0,\infty)$, and put $f=\chi_E$. Then 
    \begin{equation*}
        (K_k f)^*(t)=\left(\int_E k(s,\cdot) \, ds\right)^*(t) \leq C \int_0^{\frac{1}{\varphi(t)}} \chi_{(0,|E|)} \leq C |E|, \ \text{a.e.}\ t\in (0,\infty).
    \end{equation*}
    Because both sides of the inequality are right-continuous, we may extend this inequality to all $t\in (0,\infty)$ and, more importantly, to $t=0$. By~\eqref{rear_l_infty}, this gives
    \begin{equation*}
        \frac{1}{|E|} \int_E k(s,t) \, ds \leq C, \ \text{a.e.}\ t\in (0,\infty).
    \end{equation*}
    By Lebesgue's differentiation theorem, we have that
    \begin{equation*}
        k(x,t)=\lim_{\substack{|E|\rightarrow 0_+ \\ E \ni x}}\frac{1}{|E|} \int_E k(s,t) \, ds \leq C, \ \text{a.e.}\ t\in (0,\infty),
    \end{equation*}
    for almost every $x\in (0,\infty)$. From here, \eqref{Linf_est} follows immediately.
\end{proof}

Finally, we present a theorem, which combines all results that we have shown so far.
\begin{thm}\label{estimate_equivalence}
    Let $k$ be an admissible kernel, and let $\varphi\colon (0,\infty)\rightarrow (0,\infty)$ be a nondecreasing function and $C \in (0,\infty)$. Define $w(t):= \norm{k(\cdot,t)}_{L^1}$, $t\in (0,\infty)$.
    Then the following statements are equivalent:
    \begin{enumerate}[label=(\roman*)]
        \item   $
                \begin{aligned}[t]
                    w^*(t)\leq \frac{C}{\varphi(t)}, \ \text{a.e.}\ t\in (0,\infty) \textrm{, and } 
                    \norm{k(\cdot,t)}_{L^\infty}\leq C, \ \text{a.e.}\ t\in (0,\infty),
                \end{aligned}
                $
        \item   $K_k\colon L^\infty \rightarrow m_\varphi$, $K_k\colon L^1\rightarrow~L^\infty$, $\norm{K_k}_{L^\infty \rightarrow m_\varphi}\leq C$, and 
                $\norm{K_k}_{L^1\rightarrow L^\infty}\leq C$,
        \item   $
                \begin{aligned}[t]
                    (K_k f)^*(t)\leq C \int_0^{\frac{1}{\varphi(t)}} f^*, \ f\in L^1 + L^\infty, \ \text{a.e.}\ t\in (0,\infty).
                \end{aligned}
                $
    \end{enumerate}
\end{thm}
\begin{proof}
    We simply use Proposition~\ref{kernel_conditions} and Theorems~\ref{kernel_estimate},~\ref{necessary_condtions}.
\end{proof}

From Theorem~\ref{estimate_equivalence}, we may obtain the following corollary by using the `optimal' values for $\varphi$ and $C$.

\begin{cor}\label{sufficient_conditions}
Let $k$ be an admissible kernel. Define $w(t):= \norm{k(\cdot,t)}_{L^1}$, $t\in (0,\infty)$. Suppose that $C:=\esssup_{t\in (0,\infty)} \norm{k(\cdot,t)}_{L^\infty}<\infty$. Then 
\begin{equation*}
    (K_k f)^*(t)\leq C \int_0^{\frac{w^*(t)}{C}} f^*, \ f\in L^1 + L^\infty, \ \text{a.e.}\ t\in (0,\infty).
\end{equation*}
\end{cor}

\begin{rem}
Later, we often work with kernels for which $w$ is nonincreasing and right-continuous. This means that we can work with $w$ directly instead of using $w^*$.
\end{rem}

\section{Properties of integral operators having a variable upper bound}\label{variable_integral}
We now turn our attention to extending the principal result of~\cite{buriankova2017optimal} to the context of more general operators. 
Namely, we find optimal endpoints for the Calderón-type operators that appeared in Theorem~\ref{kernel_estimate}. We first define the main operator of interest.

\begin{defn}\label{def01:1}
    Let $b\colon (0,\infty)\rightarrow[0,\infty)$ be a nonincreasing function that is not identically zero. 
    We define $T_{b}\colon \M(0,\infty)\rightarrow\M_+(0,\infty)$ as
    \begin{equation*}
        T_{b}f(t):=\int_{0}^{b(t)}f^*,
    \end{equation*}
    for every $f\in \M(0,\infty)$, $t\in (0,\infty)$.
\end{defn}

We want to show that we can derive a new rearrangement-invariant space from an existing one using this operator. 
In the following theorem, we prove that it is indeed possible, at least under certain mild restrictions.

\begin{thm}\label{thm01:1}
    Let $b\colon (0,\infty)\rightarrow[0,\infty)$ be a nonincreasing function that is not identically zero.
    Let $Y$ be a quasinormed rearrangement-invariant space such that
    \begin{equation}\label{eq:min}
        \min\{1,b\}\in Y.
    \end{equation}
    Then the~functional $\rho$ defined by
    \begin{equation*}
        \rho (f)= \lVert T_{b}f\rVert_Y, \ f\in\mathcal{M}_+(0,\infty),
    \end{equation*}
    is a rearrangement-invariant quasinorm, additionally satisfying (P5). Moreover, if~\eqref{eq:min} is not true, $\rho$ is not a rearrangement-invariant quasinorm.
\end{thm}

\begin{proof}
    It is necessary to verify all rearrangement-invariant quasinorm axioms, namely (Q1), (P2)-(P4), and (P6). The inequality in (Q1) can be easily proven using~\eqref{fake_subaditivity}, 
    and the rest of (Q1) is trivial. (P2) follows from the monotonicity of the nonincreasing rearrangement and (P2) for $Y$. 
    (P3) can be obtained using the monotone convergence theorem. (P6) is trivial. 
    For (P4), we note that $\rho (\chi_{(0,1)})<\infty$. 
    This follows from~\eqref{eq:min} and the~fact that
    \begin{equation*}
        T_{b}(\chi_{(0,1)})(t)=\min\{1,b\}, \ t\in (0,\infty).
    \end{equation*}
    The~boundedness of the~dilation operator on quasinormed rearrangement-invariant spaces implies that $\rho (\chi_{(0,a)})<\infty$ for every $a\in (0,\infty)$. Now,
    let $E\subset (0,\infty)$ be such that $\lvert E \rvert \in (0,\infty)$. Because $\chi_E^* = \chi_{(0,\lvert E \rvert)}$, we obtain
    \begin{equation*}
        \rho (\chi_E)=\rho (\chi_E^*)=\rho(\chi_{(0,\lvert E \rvert)}) < \infty.
    \end{equation*}
    This proves (P4) and the necessity of~\eqref{eq:min}. As for proving (P5), choose $f\in\mathcal{M}_+(0,\infty)$. We first note that by the~Hardy-Littlewood inequality, we have
    \begin{equation}\label{eq:hl}
        \int_{E}f \leq \int_{0}^{\lvert E \rvert} f^{*},
    \end{equation}
    where $E\subset (0,\infty)$ and $\lvert E \rvert \in (0,\infty)$. Furthermore, we note that $A:=\lim_{t\rightarrow 0_+}b(t)$ exists, and $A\in (0,\infty]$ because $b$ is nonincreasing and $b\not\equiv 0$. 
    We will consider two cases. First, we choose $E\subset (0,\infty)$ satisfying $\lvert E \rvert\in (0,A)$.
    By the properties of $A$, we find $\xi\in (0,\infty)$ such that $b(\xi)>\abs{E}$. It follows that
    \begingroup
        \addtolength{\jot}{0.3em}
        \begin{equation*}
            \begin{aligned}
                \rho(f)&= \left\lVert \int_{0}^{b(t)}f^{*} \, \right\rVert_Y 
                \stackrel{\mathclap{\text{(P2)}}}{\geq}\left\lVert \chi_{(0,\xi)}(t) \int_{0}^{b(t)}f^{*} \, \right\rVert_Y \\
                &\geq \lVert \chi_{(0,\xi)} \rVert_Y \int_{0}^{b(\xi)}f^{*} 
                \geq \lVert \chi_{(0,\xi)} \rVert_Y \int_{0}^{\lvert E \rvert}f^{*} 
                \stackrel{\mathclap{\eqref{eq:hl}}}{\geq} \lVert \chi_{(0,\xi)} \rVert_Y \int_{E}f,
            \end{aligned}
        \end{equation*}
        hence,
        \begin{equation*}
            \int_{E}f \leq C_E \, \rho(f),
        \end{equation*}
        where $C_E=1/\lVert \chi_{(0,\xi)} \rVert_Y$. If $A=\infty$, the proof is now complete. Now, supposing $A$ is finite, we consider $\lvert E \rvert \in [A,\infty)$. Using the previous case and the fact that 
        the integral mean of a nonincreasing positive function is nonincreasing, we get that
        \begin{equation*}
            \begin{aligned}
                \int_{E}f &\leq \int_{0}^{\lvert E \rvert}f^{*} \leq \frac{2\abs{E}}{A} \, \int_{0}^{A/2} f^{*} 
                 \leq \rho(f^*) \, \frac{2\abs{E}}{A} \, C_{(0,A/2)} 
                \stackrel{\mathclap{\text{(P6)}}}{=} \rho(f) \, \frac{2\abs{E}}{A} \, C_{(0,A/2)}.
            \end{aligned}
        \end{equation*}
    \endgroup
\end{proof}

Now, we may define the above-mentioned derived quasinormed rearrangement-invariant space.

\begin{defn}\label{def01:3}
    Let $b\colon (0,\infty)\rightarrow[0,\infty)$ be a nonincreasing function that is not identically zero. Let $X$ be a 
    quasinormed rearrangement-invariant space such that
    \begin{equation}
        \min\{1,b\}\in X.
    \end{equation}
    Consider the rearrangement-invariant quasinorm $\rho$ defined by
    \begin{equation*}
        \rho (f)= \lVert T_{b}f\rVert_X, \ f\in\mathcal{M}_+(0,\infty).
    \end{equation*}
    We then define the space $\Upsilon_{b,X}:=\{f\in \mathcal{M}(0,\infty): \rho(\abs{f})<\infty\}$.
\end{defn}

\begin{rem}\label{rem01:2}
    Theorem~\ref{thm01:1} asserts that $\Upsilon_{b,X}$ is a quasinormed rearrangement-invariant 
    space and its norm satisfies (P5). Additionally, if $X$ is a rearrangement-invariant 
    space, $\Upsilon_{b,X}$ is also a rearrangement-invariant space.
\end{rem}

\begin{cnv}
    In this section, we will denote $\lim_{t\rightarrow 0_+} b(t)$ as $b(0)$.
\end{cnv}
 
\begin{rem}\label{rem01:1}
    It will be useful to note that $\{(t,s): t\in (0,\infty) \textrm{ and } s\in (0,b(t))\}=\{(t,s): s\in(0,b(0)) \textrm{ and } t\in (0,b^{-1}(s))\}$ 
    when $b$ is a nonincreasing function that is not identically zero.
\end{rem}

The following well-known duality property of $T_b$ will help us in proving our final result. 

\begin{lemma}\label{prop01:1}
    Let $b\colon (0,\infty)\rightarrow[0,\infty)$ be a nonincreasing function that is not identically zero. Suppose $X$ and $Y$ are rearrangement-invariant spaces.  
    Then $T_{b}\colon X\rightarrow Y$ if and only if $T_{b^{-1}}\colon Y'\rightarrow X'$. Moreover, $\norm{T_b}_{X\rightarrow Y}=\norm{T_{b^{-1}}}_{Y'\rightarrow X'}$.
\end{lemma}

\begin{proof}
    \begingroup
        \addtolength{\jot}{0.3em}
        The Fubini theorem combined with Remark~\ref{rem01:1} yields
        \begin{equation*}
            \begin{aligned}
                \int_{0}^{\infty} h^* \, T_{b}g
                & =  \int_{0}^{\infty}\int_{0}^{b(t)} h^*(t)\,g^*(s)\,ds\, dt \\
                & = \int_{0}^{\infty}\int_{0}^{b^{-1}(s)} h^*(t) \,g^*(s)\,dt\, ds \\
                & =  \int_{0}^{\infty} g^* \, T_{b^{-1}}h,
            \end{aligned}
        \end{equation*}
        for every $g$, $h\in \mathcal{M}(0,\infty)$. We note that $T_{b}f=(T_{b}f)^*$ a.e.\ because $T_{b}f$ is nonincreasing. 
        This fact, interchanging suprema,~\eqref{abs_out}, and the equality in~\eqref{associate_ineq} now show that
        \begin{equation*}
            \begin{aligned}
                \norm{T_b}_{X\rightarrow Y} &= \sup_{\norm{g}_X \leq 1} \norm{T_{b}g}_Y 
                = \sup_{\norm{g}_X \leq 1} \norm{T_{b}g}_{Y''} \\
                & = \sup_{\norm{g}_X \leq 1} \sup_{\norm{h}_{Y'} \leq 1} \int_{0}^{\infty} h^* \, T_{b}g \\
                & = \sup_{\norm{h}_{Y'} \leq 1} \sup_{\norm{g}_X \leq 1} \int_{0}^{\infty} g^* \, T_{b^{-1}}h \\
                & = \sup_{\norm{h}_{Y'} \leq 1} \norm{T_{b^{-1}}h}_{X'} = \norm{T_{b^{-1}}}_{Y'\rightarrow X'}.
            \end{aligned}
        \end{equation*}
    \endgroup
\end{proof}

Finally, we state the first part of this section's principal result. For a given rearrangement-invariant space, we find its optimal range partner 
with respect to a given operator $T_b$.

\begin{thm}\label{thm01:2}
    Let $b\colon (0,\infty)\rightarrow[0,\infty)$ be a nonincreasing function that is not identically zero. Let $X$ be a 
    rearrangement-invariant space such that
    \begin{equation}\label{eq:min2}
        \min\{1,b^{-1}\}\in X'.
    \end{equation}
    Then $T_{b}\colon X\rightarrow (\Upsilon_{b^{-1},X'})'$. Furthermore, the space $(\Upsilon_{b^{-1},X'})'$ is the optimal range partner 
    of $X$ with respect to $T_{b}$ in the class of rearrangement-invariant spaces, and no range partner in this class exists if~\eqref{eq:min2} is false.
\end{thm}

\begin{proof}
    Choose $f\in X$. 
    We note that $T_{b}f=(T_{b}f)^*$ a.e.\ because $T_{b}f$ is nonincreasing. 
    By this fact,~\eqref{abs_out}, and the Fubini theorem combined with Remark~\ref{rem01:1}, we have
    \begin{equation*}
        \begin{aligned}
            \lVert T_{b}f \rVert_{(\Upsilon_{b^{-1},X'})'} &= \sup_{\lVert g \rVert_{\Upsilon_{b^{-1},X'}} \leq 1} \int_{0}^{\infty} g^* \,T_{b}f \\
            &= \sup_{\lVert g \rVert_{\Upsilon_{b^{-1},X'}} \leq 1} \int_{0}^{\infty}\int_{0}^{b(t)} g^*(t)\,f^*(s)\,ds\, dt \\
            &= \sup_{\lVert g \rVert_{\Upsilon_{b^{-1},X'}} \leq 1} \int_{0}^{\infty}\int_{0}^{b^{-1}(s)} g^*(t) \,f^*(s)\,dt\, ds \\
            &= \sup_{\lVert g \rVert_{\Upsilon_{b^{-1},X'}} \leq 1} \int_{0}^{\infty} T_{b^{-1}}g\,f^*.
        \end{aligned}
    \end{equation*}
    Finally, using the H\"older inequality, we get
    \begin{equation*}
        \begin{aligned}
            \lVert T_{b}f \rVert_{(\Upsilon_{b^{-1},X'})'} &= \sup_{\lVert g \rVert_{\Upsilon_{b^{-1},X'}} \leq 1} \int_{0}^{\infty} T_{b^{-1}}g\,f^* \\
            &\leq \sup_{\lVert g \rVert_{\Upsilon_{b^{-1},X'}} \leq 1} \lVert T_{b^{-1}}g \rVert_{X'} \, \lVert f^* \rVert_{X} \\
            &= \sup_{\lVert g \rVert_{\Upsilon_{b^{-1},X'}} \leq 1} \lVert g \rVert_{\Upsilon_{b^{-1},X'}} \, \lVert f \rVert_{X} = \lVert f \rVert_{X},
        \end{aligned}
    \end{equation*}
    which proves that $T_{b}\colon X\rightarrow (\Upsilon_{b^{-1},X'})'$. 
    For optimality, suppose that $Z$ is a rearrangement-invariant space such that $T_{b}\colon X\rightarrow Z$.
    By Lemma~\ref{prop01:1}, we have that $T_{b^{-1}}\colon Z'\rightarrow X'$. This means that for some~$C \in (0,\infty)$
    \begin{equation*}
        \lVert f \rVert_{\Upsilon_{b^{-1},X'}} = \lVert T_{b^{-1}}f \rVert_{X'} \leq  C \,
         \lVert f \rVert_{Z'},
    \end{equation*}
    for all $f \in Z'$. Therefore, $Z' \hookrightarrow \Upsilon_{b^{-1},X'}$, and $(\Upsilon_{b^{-1},X'})'\hookrightarrow Z$. Thus, $(\Upsilon_{b^{-1},X'})'$ is optimal.
    Next, consider that $\min\{1,b^{-1}\}\notin X'$, and let $Y$ be a rearrangement-invariant space such that $T_{b}\colon X\rightarrow Y$. Again, by Lemma~\ref{prop01:1},
    we have for some~$C \in (0,\infty)$
    \begin{equation*}
        \lVert T_{b^{-1}}f \rVert_{X'} \leq C \, \lVert f \rVert_{Y'},
    \end{equation*}
    for all $f\in Y'$. Hence,
    \begin{equation*}
        \begin{aligned}
            \lVert \chi_{(0,1)} \rVert_{Y'} & \geq C^{-1} \, \lVert T_{b^{-1}}\chi_{(0,1)} \rVert_{X'} 
            = C^{-1} \, \lVert \min\{1,b^{-1}\} \rVert_{X'} = \infty.
        \end{aligned}
    \end{equation*}
    Therefore, $Y'$ is not a rearrangement-invariant space, which contradicts our assumption on $Y$.
\end{proof}

In the second part of the principal result, we find the optimal domain partner for a given quasinormed rearrangement-invariant space with respect to 
a given operator $T_b$. It is of note that, unlike the last result, this holds for general quasinormed rearrangement-spaces because the proof does not utilize any 
duality properties.

\begin{thm}\label{thm01:4}
    Let $b\colon (0,\infty)\rightarrow[0,\infty)$ be a nonincreasing function that is not identically zero.
    Let $Y$ be a quasinormed rearrangement-invariant space such that
    \begin{equation}\label{eq:min3}
        \min\{1,b\}\in Y.
    \end{equation}
    Then $T_{b}\colon \Upsilon_{b,Y}\rightarrow Y$. 
    Furthermore, the space $\Upsilon_{b,Y}$ is the optimal domain partner of $Y$ with respect to $T_{b}$, and if~\eqref{eq:min3} is false, there is no domain partner in the class of quasinormed rearrangement-invariant spaces.
\end{thm}

\begin{proof}
    Choose $f\in \Upsilon_{b,Y}$.
    From the definition of $\Upsilon_{b,Y}$, we have
    \begin{equation*}
        \lVert T_{b}f \rVert_Y = \lVert f \rVert_{\Upsilon_{b,Y}},
    \end{equation*}
    which means that $T_{b}\colon \Upsilon_{b,Y}\rightarrow Y$. Next, consider a quasinormed rearrangement-invariant space $Z$ such that $T_{b}\colon Z\rightarrow Y$.
    Then for some $C \in (0,\infty)$ we have
    \begin{equation*}
        \lVert f \rVert_{\Upsilon_{b,Y}} = \lVert T_{b}f \rVert_Y \leq C \lVert f \rVert_Z,
    \end{equation*}
    for every $f\in Z$. Hence, $Z\hookrightarrow \Upsilon_{b,Y}$, which proves the optimality. Finally, suppose that $\min\{1,b\}\notin Y$ and there exists 
    a quasinormed rearrangement-invariant space $X$ such that $T_{b}\colon X\rightarrow Y$. Then for some~$C \in (0,\infty)$
    \begin{equation*}
        \begin{aligned}
            \lVert \chi_{(0,1)} \rVert_{X} & \geq C \, \lVert T_{b}\chi_{(0,1)} \rVert_{Y} = C \, \lVert \min\{1,b\} \rVert_{Y} = \infty.
        \end{aligned}
    \end{equation*}
    Therefore, $X$ is not a quasinormed rearrangement-invariant space, which contradicts our assumption on $X$.
\end{proof}

\section{Applications}\label{examples}
Lastly, we present examples where we attempt to apply the theory we developed in the previous two sections. 
Our primary focus will be on investigating options for combining Theorem~\ref{estimate_equivalence}, Theorem~\ref{thm01:2}, and Theorem~\ref{thm01:4} to obtain optimal endpoint spaces for 
the previously discussed integral kernel operators. We will see that the strategy from~\cite{buriankova2017optimal} has its shortcomings when it comes to generalizing to other kernel operators. \par
First, we present a class of operators for which we are able to find their
optimal endpoints.

\begin{example}\label{ideal_case}
    Let $\psi\colon (0,\infty) \rightarrow (0,\infty)$, $\psi \in L^1 \cap L^\infty$, be nonincreasing and $\tilde{\varphi}\colon (0,\infty) \rightarrow (0,\infty)$ be nondecreasing and right-continuous. 
    Define $k(s,t):=\psi(s \tilde{\varphi}(t))$. Then
    \begin{equation*}
        \esssup_{t\in (0,\infty)} \norm{k(\cdot,t)}_{L^\infty}=\norm{\psi}_{L^\infty},
    \end{equation*}
    \begin{equation*}
        w(t):=\norm{k(\cdot,t)}_{L^1}=\frac{\norm{\psi}_{L^1}}{\tilde{\varphi}(t)}, \ t\in (0,\infty).
    \end{equation*}
    Set
    \begin{equation*}
        b(t):=\frac{w(t)}{\norm{\psi}_{L^\infty}}, \ t\in (0,\infty).
    \end{equation*}
    Corollary~\ref{sufficient_conditions} 
    provides a Calderón-type estimate for $K_k$,
    \begin{equation*}
        (K_kf)^*(t)\leq \norm{\psi}_{L^\infty} \int_0^{b(t)} f^* = \norm{\psi}_{L^\infty} T_{b}f(t),\ \text{a.e.} \ t\in (0,\infty),
    \end{equation*}
    for all $f\in L^1 + L^\infty$. In this case, the estimate can be reversed to some extent,
    \begin{equation}\label{ideal_lower}
        \begin{aligned}
            K_k(f^*)(t)&\geq \int_0^{\frac{\norm{\psi}_{L^1}}{\norm{\psi}_{L^\infty} \tilde{\varphi}(t)}} f^*(s) \psi(s \tilde{\varphi}(t)) \, ds \\ &\geq
        \psi\left(\frac{\norm{\psi}_{L^1}}{\norm{\psi}_{L^\infty}}\right) \int_0^{\frac{\norm{\psi}_{L^1}}{\norm{\psi}_{L^\infty} \tilde{\varphi}(t)}} f^*, \ t\in (0,\infty),
        \end{aligned}
    \end{equation}
    for all $f\in \M(0,\infty)$. With this information, it is simple to show that endpoint optimality for the operator $T_{b}$ is transferred to $K_k$. This allows  
    us to use Theorem~\ref{thm01:2} and Theorem~\ref{thm01:4} to find optimal endpoints for $K_k$ because $b$ is nonincreasing. Explicit examples of such kernels
    include
    \begin{equation*}
        k(s,t)=e^{-st} \ \textrm{(Laplace transform)},
    \end{equation*}
    \begin{equation*}
        k(s,t)=\frac{1}{(st)^p+c}, \ p>1, \ c \in (0,\infty).
    \end{equation*}
\end{example}

In contrast to Example~\ref{ideal_case}, we show an operator for which we may not apply the theory.

\begin{example}
    Let
    \begin{equation*}
        k(s,t):=\frac{1}{s^p + t^p}, \ p>0.    
    \end{equation*}
    Then $\esssup_{t\in (0,\infty)} \norm{k(\cdot,t)}_{L^\infty}=\infty$. Hence, by Theorem~\ref{estimate_equivalence}, there exist no $C \in (0,\infty)$ and $\varphi\colon (0,\infty)\rightarrow(0,\infty)$ such that 
    an inequality of the form~\eqref{estimate} would hold.
\end{example}

We will attempt to apply the theory to a similar operator. However, we first need the following simple duality lemma for symmetric kernels.

\begin{lemma}\label{symmetric_kernel}
Let $k$ be a nonincreasing admissible kernel, and let $X$, $Y$ be rearrangement-invariant spaces. 
Suppose that $k(s,t)=k(t,s)$ for every $t$, $s\in (0,\infty)$.
Then $K_k\colon X \rightarrow Y$ if and only if $K_k\colon Y' \rightarrow X'$. 
\end{lemma}
\begin{proof}
Suppose that $K_k\colon X \rightarrow Y$. By \eqref{abs_out}, the Hardy-Littlewood inequality, the Fubini theorem, interchanging suprema, the fact that supremum over nonincreasing nonnegative functions is less than 
supremum over all functions, and finally the fact that $Y=Y''$, we have
\begin{equation*}
\begin{aligned}
\norm{K_k}_{Y'\rightarrow X'}=\sup_{\norm{g}_{Y'}\leq 1} \norm{K_k g}_{X'}&=\sup_{\norm{g}_{Y'}\leq 1} \sup_{\norm{f}_{X}\leq 1} \int_0^\infty f^* (t) \left(\int_0^\infty g(s) k(s,\cdot)  \, ds\right)^*(t) \, dt \\
&\leq \sup_{\norm{g}_{Y'}\leq 1} \sup_{\norm{f}_{X}\leq 1} \int_0^\infty f^* (t) \int_0^\infty g^*(s) k(s,t)  \, ds \, dt \\
&=\sup_{\norm{f}_{X}\leq 1} \sup_{\norm{g}_{Y'}\leq 1} \abs{\int_0^\infty g^* (s) \int_0^\infty  f^*(t) k(t,s) \, dt \, ds} \\
&\leq \sup_{\norm{f}_{X}\leq 1} \sup_{\norm{g}_{Y'}\leq 1} \int_0^\infty \abs{g(s) \int_0^\infty  f(t) k(t,s) \, dt} \, ds \\
&= \sup_{\norm{f}_{X}\leq 1} \norm{K_k f}_Y = \norm{K_k}_{X\rightarrow Y}
\end{aligned}
\end{equation*}
The proof of the other implication is analogous.
\end{proof}

We now try to `fix' the operator from the previous example by perturbing the denominator. However, in the end, we show that the proposed optimal endpoints are, 
in fact, not generally optimal.

\begin{example}
    Let
    \begin{equation*}
        k(s,t):=\frac{1}{s^p + t^p+c}, \ p\in (1,\infty), \ c \in (0,\infty).    
    \end{equation*}
    One can calculate that 
    \begin{equation*}
        w(t):=\norm{k(\cdot,t)}_{L^1}=\frac{\pi}{p \sin\left({\frac{\pi}{p}}\right)} (c+t^p)^{\frac{1}{p}-1}, \ t\in (0,\infty),
    \end{equation*}
    and
    \begin{equation*}
        \esssup_{t\in (0,\infty)} \norm{k(\cdot,t)}_{L^\infty}=\frac{1}{c}.
    \end{equation*}
    Set
    \begin{equation*}
         b(t):=c w(t), \ t\in (0,\infty).
    \end{equation*}
    Corollary~\ref{sufficient_conditions} provides the estimate
    \begin{equation*}
        (K_k f)^*(t)\leq \frac{1}{c} \int_0^{b(t)} f^*, \ t\in (0,\infty),
    \end{equation*}
    for all $f\in L^1 + L^\infty$. To disprove endpoint optimality, consider the special case 
    \begin{equation*}
        k(s,t):=\frac{1}{s^2 + t^2+1}.    
    \end{equation*}
    Then 
    \begin{equation*}
        b(t)=\norm{k(\cdot,t)}_{L^1}=\frac{\pi}{2\sqrt{1+t^2}}, \ t\in (0,\infty).
    \end{equation*}
    We calculate that $b^{-1}(t)=\sqrt{\frac{\pi^2}{4 t^2}-1}, \ t\in\left(0,\frac{\pi}{2}\right)$ and $b^{-1}(t)=0,\ t\in [\frac{\pi}{2},\infty)$.
    Aiming for a contradiction, assume that $T_b$ shares its optimal endpoints with $K_k$. We have that $\min\{1,b\}\in L^2$ and 
    $\min\{1,b^{-1}\}\in L^2$. Because of our assumption and Theorem~\ref{thm01:2}, $\Upsilon'_{b^{-1},L^2}$ is the optimal range partner 
    for $L^2$. Using Lemma~\ref{symmetric_kernel}, we observe that $\Upsilon_{b^{-1},L^2}$ is the optimal domain partner for $L^2$ among rearrangement-invariant spaces. 
    By Theorem~\ref{thm01:4}, $\Upsilon_{b,L^2}$ is also the optimal domain partner for $L^2$. Thus, $\Upsilon_{b^{-1},L^2}\approx 
    \Upsilon_{b,L^2}$, because both are rearrangement-invariant spaces. However, one can easily verify that $\norm{f}_{\Upsilon_{b^{-1},L^2}}<\infty$ and 
    $\norm{f}_{\Upsilon_{b,L^2}}=\infty$ for $f(t):=t^{-3/4},\ t\in (0,\infty)$, so $\Upsilon_{b^{-1},L^2}\not\approx\Upsilon_{b,L^2}$. This means that 
    $K_k$ and $T_b$ do not share their optimal endpoints. 
\end{example}

We have seen that the results of using our methods to find optimal endpoints for operators $K_k$ vary. Furthermore, we have demonstrated that the case of the Laplace transform in~\cite{buriankova2017optimal} and the operators in Example~\ref{ideal_case} is fairly unique because the Calder\'on-type estimate can be reversed. To conclude, we state a simple sufficient condition for the reversal of our Calder\'on-type estimate. This then guarantees that $K_k$ and $T_b$ share their optimal endpoints.

\begin{prop}
    Let $k$ be an admissible kernel nonincreasing in the first variable. Let $k$ satisfy the conditions of Corollary~\ref{sufficient_conditions} and denote 
    $b(t):=w^*(t)/C$, $t\in (0,\infty)$. Define
    \begin{equation}\label{sufficient_constant}
        c:=\essinf_{t\in (0,\infty)} \frac{1}{b(t)} \int_0^{b(t)} k(s,t) \, ds.
    \end{equation}
    If $c \in (0,\infty)$, every $f\in \M(0,\infty)$ satisfies
    \begin{equation*}
        \frac{c}{2} \int_0^{b(t)} f^* \leq \int_0^\infty  f^*(s) k(s,t) \, ds, \ a.e.\ t\in (0,\infty).
    \end{equation*}
    \begin{proof}
        By~\eqref{sufficient_constant}, we have
        \begin{equation*}
            \begin{aligned}
                c \int_0^{b(t)} f^* &\leq \frac{1}{b(t)} \int_0^{b(t)} k(s,t) \, ds \int_0^{b(t)} f^* = \frac{1}{b(t)} \int_0^{b(t)} \int_0^{b(t)} f^*(u) k(s,t) \, ds \, du \\
                &\leq\frac{1}{b(t)} \left(\int_0^{b(t)} \int_0^{b(t)}\left( \chi_{\{s>u\}}  f^*(u) k(u,t) +\chi_{\{u>s\}} f^*(s) k(s,t) \right)\, ds \, du\right) \\
                &= \frac{2}{b(t)} \left( \int_0^{b(t)} \int_0^{b(t)}  f^*(s) k(s,t) \, ds \, du \right) \leq 2 \int_0^\infty  f^*(s) k(s,t) \, ds, \ a.e.\ t\in (0,\infty).
            \end{aligned}
        \end{equation*}
    \end{proof}
\end{prop}

\begin{rem}
    For the Laplace transform,~\eqref{sufficient_constant} evaluates to $c=1-\frac{1}{e}$.
\end{rem}

\bibliography{bibliography}

@book {pick2012function,
    AUTHOR = {Pick, L. and Kufner, A. and John, O. and
              Fu\v{c}\'{i}k, S.},
     TITLE = {Function spaces. {V}ol. 1},
    SERIES = {De Gruyter Series in Nonlinear Analysis and Applications},
    VOLUME = {14},
   EDITION = {extended},
 PUBLISHER = {Walter de Gruyter \& Co., Berlin},
      YEAR = {2013},
     PAGES = {xvi+479},
      ISBN = {978-3-11-025041-1; 978-3-11-025042-8},
   MRCLASS = {46Exx (46-02)},
  MRNUMBER = {3024912},
MRREVIEWER = {Dumitru\ Popa},
}

@BOOK{bs88,
	address = {Boston},
	edition = {Pure and applied mathematics, 129},
	title = {Interpolation of operators},
	isbn = {0-12-088730-4},
	publisher = {Academic Press, Inc., Boston, MA},
	author = {Bennett, C. and Sharpley, R. C.},
	year = {1988}
}

@article {buriankova2017optimal,
    AUTHOR = {Buri\'{a}nkov\'{a}, E. and Edmunds, D. E. and Pick, L.},
     TITLE = {Optimal function spaces for the {L}aplace transform},
   JOURNAL = {Rev. Mat. Complut.},
  FJOURNAL = {Revista Matem\'atica Complutense},
    VOLUME = {30},
      YEAR = {2017},
    NUMBER = {3},
     PAGES = {451--465},
      ISSN = {1139-1138,1988-2807},
   MRCLASS = {46E30 (26D10 44A10 46B70 47B38)},
  MRNUMBER = {3690749},
MRREVIEWER = {Francisco\ L.\ Hern\'andez},
       DOI = {10.1007/s13163-017-0234-5},
       URL = {https://doi.org/10.1007/s13163-017-0234-5},
}

@article{Nekvinda_2024,
AUTHOR = {Nekvinda, A. and Pe\v{s}a, D.},
     TITLE = {On the properties of quasi-{B}anach function spaces},
   JOURNAL = {J. Geom. Anal.},
  FJOURNAL = {Journal of Geometric Analysis},
    VOLUME = {34},
      YEAR = {2024},
    NUMBER = {8},
     PAGES = {Paper No. 231, 29},
      ISSN = {1050-6926,1559-002X},
   MRCLASS = {46A16 (46E30)},
  MRNUMBER = {4746146},
       DOI = {10.1007/s12220-024-01673-y},
       URL = {https://doi.org/…3-y},  
}

@article {Optimal_Sobolev,
    AUTHOR = {Edmunds, D. E. and Kerman, R. and Pick, L.},
     TITLE = {Optimal {S}obolev imbeddings involving rearrangement-invariant
              quasinorms},
   JOURNAL = {J. Funct. Anal.},
  FJOURNAL = {Journal of Functional Analysis},
    VOLUME = {170},
      YEAR = {2000},
    NUMBER = {2},
     PAGES = {307--355},
      ISSN = {0022-1236,1096-0783},
   MRCLASS = {46E35 (46E30)},
  MRNUMBER = {1740655},
MRREVIEWER = {Maria\ A.\ Ragusa},
       DOI = {10.1006/jfan.1999.3508},
       URL = {https://doi.org/10.1006/jfan.1999.3508},
}

@article {classical_lorentz,
    AUTHOR = {\'{C}wikel, M. and Kami\'{n}ska, A. and Maligranda, L.
              and Pick, L.},
     TITLE = {Are generalized {L}orentz ``spaces'' really spaces?},
   JOURNAL = {Proc. Amer. Math. Soc.},
  FJOURNAL = {Proceedings of the American Mathematical Society},
    VOLUME = {132},
      YEAR = {2004},
    NUMBER = {12},
     PAGES = {3615--3625},
      ISSN = {0002-9939,1088-6826},
   MRCLASS = {46E30},
  MRNUMBER = {2084084},
MRREVIEWER = {Oscar\ Blasco},
       DOI = {10.1090/S0002-9939-04-07477-5},
       URL = {https://doi.org/10.1090/S0002-9939-04-07477-5},
}

@article {MR1758860,
    AUTHOR = {Cianchi, A. and Kerman, R. and Opic, B. and Pick, L.},
     TITLE = {A sharp rearrangement inequality for the fractional maximal
              operator},
   JOURNAL = {Studia Math.},
  FJOURNAL = {Studia Mathematica},
    VOLUME = {138},
      YEAR = {2000},
    NUMBER = {3},
     PAGES = {277--284},
      ISSN = {0039-3223,1730-6337},
   MRCLASS = {42B25 (47G10)},
  MRNUMBER = {1758860},
MRREVIEWER = {Wo-Sang\ Young},
}

@article {MR146673,
    AUTHOR = {O'Neil, R.},
     TITLE = {Convolution operators and {$L(p,\,q)$} spaces},
   JOURNAL = {Duke Math. J.},
  FJOURNAL = {Duke Mathematical Journal},
    VOLUME = {30},
      YEAR = {1963},
     PAGES = {129--142},
      ISSN = {0012-7094,1547-7398},
   MRCLASS = {47.25 (46.35)},
  MRNUMBER = {146673},
MRREVIEWER = {I.\ I.\ Hirschman, Jr.},
       URL = {http://projecteuclid.org/euclid.dmj/1077374532},
}

@article {propertiesrearrangement,
    AUTHOR = {Musilov\'{a}, A. and Nekvinda, A. and Pe\v{s}a, D.
              and Tur\v{c}inov\'{a}, H.},
     TITLE = {On the properties of rearrangement-invariant quasi-{B}anach
              function spaces},
   JOURNAL = {Nonlinear Anal.},
  FJOURNAL = {Nonlinear Analysis. Theory, Methods \& Applications. An
              International Multidisciplinary Journal},
    VOLUME = {260},
      YEAR = {2025},
     PAGES = {Paper No. 113854, 30},
      ISSN = {0362-546X,1873-5215},
   MRCLASS = {46E30 (46A16)},
  MRNUMBER = {4919871},
MRREVIEWER = {Francisco\ L.\ Hern\'andez},
       DOI = {10.1016/j.na.2025.113854},
       URL = {https://doi.org/10.1016/j.na.2025.113854},
}

@book {MR1625845,
    AUTHOR = {Evans, L. C.},
     TITLE = {Partial differential equations},
    SERIES = {Graduate Studies in Mathematics},
    VOLUME = {19},
 PUBLISHER = {American Mathematical Society, Providence, RI},
      YEAR = {1998},
     PAGES = {xviii+662},
      ISBN = {0-8218-0772-2},
   MRCLASS = {35-01},
  MRNUMBER = {1625845},
MRREVIEWER = {Luigi\ Rodino},
       DOI = {10.1090/gsm/019},
       URL = {https://doi.org/10.1090/gsm/019},
}

@book {MR482275,
    AUTHOR = {Bergh, J. and L\"{o}fstr\"{o}m, J.},
     TITLE = {Interpolation spaces. {A}n introduction},
    SERIES = {Grundlehren der Mathematischen Wissenschaften},
    VOLUME = {No. 223},
 PUBLISHER = {Springer-Verlag, Berlin-New York},
      YEAR = {1976},
     PAGES = {x+207},
   MRCLASS = {46M35},
  MRNUMBER = {482275},
}

@incollection {Maligranda,
    AUTHOR = {Maligranda, L.},
     TITLE = {The {$K$}-functional for symmetric spaces},
 BOOKTITLE = {Interpolation spaces and allied topics in analysis ({L}und,
              1983)},
    SERIES = {Lecture Notes in Math.},
    VOLUME = {1070},
     PAGES = {169--182},
 PUBLISHER = {Springer, Berlin},
      YEAR = {1984},
      ISBN = {3-540-13363-1},
   MRCLASS = {46E30 (46M35)},
  MRNUMBER = {760482},
MRREVIEWER = {M.\ Z.\ Berkola\u iko},
       DOI = {10.1007/BFb0099100},
       URL = {https://doi.org/10.1007/BFb0099100},
}

@article {Torchinsky,
    AUTHOR = {Torchinsky, A.},
     TITLE = {The {$K$}-functional for rearrangement invariant spaces},
   JOURNAL = {Studia Math.},
  FJOURNAL = {Polska Akademia Nauk. Instytut Matematyczny. Studia
              Mathematica},
    VOLUME = {64},
      YEAR = {1979},
    NUMBER = {2},
     PAGES = {175--190},
      ISSN = {0039-3223,1730-6337},
   MRCLASS = {46E30 (46M35)},
  MRNUMBER = {537119},
MRREVIEWER = {A.\ Kufner},
       DOI = {10.4064/sm-64-2-175-190},
       URL = {https://doi.org/10.4064/sm-64-2-175-190},
}

@article {Curbera,
    AUTHOR = {Curbera, G. P. and Ricker, W. J.},
     TITLE = {Optimal domains for kernel operators via interpolation},
   JOURNAL = {Math. Nachr.},
  FJOURNAL = {Mathematische Nachrichten},
    VOLUME = {244},
      YEAR = {2002},
     PAGES = {47--63},
      ISSN = {0025-584X,1522-2616},
   MRCLASS = {47B34 (46E30 46G10 46M35)},
  MRNUMBER = {1928916},
MRREVIEWER = {Evgenii\ I.\ Pustylnik},
       DOI = {10.1002/1522-2616(200210)244:1<47::AID-MANA47>3.0.CO;2-B},
       URL =
              {https://doi.org/10.1002/1522-2616(200210)244:1<47::AID-MANA47>3.0.CO;2-B},
}

@book {Ricker,
    AUTHOR = {Okada, S. and Ricker, W. J. and S\'anchez P\'erez,
              E. A.},
     TITLE = {Optimal domain and integral extension of operators},
    SERIES = {Operator Theory: Advances and Applications},
    VOLUME = {180},
      NOTE = {},
 PUBLISHER = {Birkh\"auser Verlag, Basel},
      YEAR = {2008},
     PAGES = {xii+400},
      ISBN = {978-3-7643-8647-4},
   MRCLASS = {47B99 (28B05 43A15 43A25 46G10)},
  MRNUMBER = {2418751},
MRREVIEWER = {J.\ Diestel},
}
\end{document}